\documentclass[10pt]{article}
\font\smallit=cmti10
\font\smalltt=cmtt10

\usepackage{amssymb,latexsym,amsmath,epsfig,amsthm,tabularx,chngpage} 

\makeatletter

\renewcommand\section{\@startsection {section}{1}{\z@}
{-30pt \@plus -1ex \@minus -.2ex}
{2.3ex \@plus.2ex}
{\normalfont\normalsize\bfseries}}

\renewcommand\subsection{\@startsection{subsection}{2}{\z@}
{-3.25ex\@plus -1ex \@minus -.2ex}
{1.5ex \@plus .2ex}
{\normalfont\normalsize\bfseries}}

\renewcommand{\@seccntformat}[1]{\csname the#1\endcsname. }

\makeatother

\newtheorem{theorem}{Theorem}
\newtheorem{lemma}{Lemma}

\newcommand{\beq}{\begin{equation}}
\newcommand{\eeq}{\end{equation}}

\def\({\left(}
\def\){\right)}

\begin{document}

\begin{center}
\uppercase{\bf  Efficient computation of terms of linear recurrence sequences of  any order}
\vskip 20pt
{\bf Dmitry I. Khomovsky}\\
{\smallit Lomonosov Moscow State University}\\
{\tt khomovskij@physics.msu.ru}
\end{center}
\vskip 30pt

\centerline{\smallit Received: , Revised: , Accepted: , Published: } 
\vskip 30pt

\centerline{\bf Abstract}

\noindent
In this paper we give efficient algorithms for computing second-, third-, and fourth-order linear recurrences. We also present an algorithm scheme  for computing  terms with the indices $N,\ldots,N+n-1$ of an $n${\it th}-order linear recurrence. Unlike Fiduccia's algorithm our approach uses certain formulas for modular polynomial squarings.

\pagestyle{myheadings}
\markright{\smalltt \hfill}
\thispagestyle{empty}
\baselineskip=12.875pt
\vskip 30pt

\section{Introduction}
Let $\{W_k(a_0,\ldots,a_{n-1};p_0\ldots p_{n-1})\}$ be an $n{\it th}$-order linear recurrence defined by the relation  \beq\label{f01}f_{k+n}=p_0f_{k+n-1}+p_1f_{k+n-2}+\ldots+p_{n-1}f_{k},\eeq
with the initial values $W_i=a_i$ $(0\leq i\leq n-1)$. The characteristic polynomial is
\beq\label{f09}
g(x)=x^n-(p_0x^{n-1}+p_1x^{n-2}+\ldots+p_{n-1}).
\eeq
A widely known particular case is the Lucas sequences $\{U_{k}(P,Q)\}$, $\{V_{k}(P,Q)\}$. They are defined recursively by
\beq \label{DefLuc}
f_{k+2}=P f_{k+1}-Q f_{k},
\eeq
with the initial values $U_{0}=0,\ U_{1}=1,\, V_{0}=2,\, V_{1}=P$. The characteristic polynomial in this case is $x^2-P x+Q$.

Computation of linear recurrences has been studied by many authors \cite{0,01,02}. The most effective algorithm was proposed by Fiduccia in 1985. To obtain the $N${\it th} term of an $n${\it th}-order linear recurrence using this method, we need to compute $r(x)=x^N\bmod{g(x)}$, where $g(x)=x^n-\sum_{i=0}^{n-1}p_ix^{n-1-i}$. Then we compute $r(C)$, where $C$ is  the $n\times n$ companion matrix of the linear recurrence:
\beq
C=\begin{pmatrix}
0 & & & &p_{n-1}\\
1 & & & &p_{n-2}\\
  &1& & &p_{n-3}\\
  & & \ddots& &\vdots\\
  & & & 1&p_{0}\\
\end{pmatrix}.
\eeq
Finally, we multiply the row vector of initial values $(a_0,\ldots,a_{n-1})$  by the first column of $r(C)$ and obtain the $N${\it th} term. The computational complexity of this algorithm is $O(\mu(n) \log N + n^3)$. Here, $\mu(n)$ is the total number of operations required to multiply two polynomials of degree $n-1$ in the polynomial ring. Fiduccia actually manages to exploit the structure of the matrix $C$ in order to decrease the complexity to $O(\mu(n)\log N)$, see  Theorem $3.1$ and Proposition $3.2$ in \cite{0}.

\section{\bf{Computation of second-order linear recurrences}}
Let  the  second-order linear recurrence sequence  $\{W_k\}$ be defined by the relation\footnote{For  recurrences of order greater than $2$  we will use the relation $(\ref{f01})$.} $W_{k+2}=P W_{k+1}-Q W_{k}$, with $W_0=A, W_1=B$.  It was intensively studied by Horadam \cite{1.11,1.12}.

For the Lucas sequences we have the following matrix formula:
\beq\label{M}
\begin{pmatrix}
U_{k+1} & V_{k+1} \\
U_{k} & V_{k}
\end{pmatrix}=M\begin{pmatrix}
U_{k} & V_{k} \\
U_{k-1} & V_{k-1}
\end{pmatrix},\,\,\, \text{ where } M=\begin{pmatrix}
P & -Q \\
1 & 0
\end{pmatrix}.
\eeq
Then
\beq\label{M1}
\begin{pmatrix}
U_{k+1} & V_{k+1} \\
U_{k} & V_{k}
\end{pmatrix}=M^{k}\begin{pmatrix}
1 & P \\
0 & 2
\end{pmatrix}.
\eeq
\begin{lemma}\label{L2}
For the sequence $\{W_k(A,B;P,Q)\}$ the following holds:
\beq\label{f06}
W_k=B U_k-A Q U_{k-1},\eeq
\beq\label{f07}
W_k=(B-A P)U_k+AU_{k+1}.\eeq\end{lemma}
\begin{proof}
We have:
\begin{align}
\begin{pmatrix}
W_{k+1}  \\
W_{k}
\end{pmatrix}&=M^{k}\begin{pmatrix}
B \\
A
\end{pmatrix}=B M^k \begin{pmatrix}
1 \\
0
\end{pmatrix}+A M^k \begin{pmatrix}
0 \\
1
\end{pmatrix}=B \begin{pmatrix}
U_{k+1} \\
U_{k}
\end{pmatrix}+A M^{k-1}\begin{pmatrix}
-Q \\
0
\end{pmatrix}\nonumber\\
&=\begin{pmatrix}
B U_{k+1} -A Q U_{k}\\
B U_{k} -A Q U_{k-1}
\end{pmatrix}.\nonumber
\end{align}
From this we get $(\ref{f06})$. By the definition of the Lucas sequence $Q U_{k-1}=PU_{k}-U_{k+1}$. Using this, we obtain $(\ref{f07})$.
\end{proof}
The obtained result is known (for example, see \cite{1.11}). We see that computation of remote terms  of $\{W_k(A,B;P,Q)\}$ can be  done by the Lucas sequence $\{U_k(P,Q)\}$.
In a sense, $\{U_k\}$ is a basic.

We note that the result given in the following theorem is  known, moreover, there is a generalization of this \cite{000}. But we still give the proof,
since we will use a similar approach for higher-order linear recurrences.
\begin{theorem}\label{T1}
Let $\{U_{k}(P,Q)\}$ be  the Lucas sequence. Then
\beq\label{f02}
\begin{pmatrix}
U_{mk+1}  \\
U_{mk}
\end{pmatrix}=\begin{pmatrix}
U_{k+1} & - Q U_{k}\\
U_{k} & U_{k+1}-P U_k
\end{pmatrix}^{m-1}\begin{pmatrix}
U_{k+1}  \\
U_{k}
\end{pmatrix}.
\eeq
\end{theorem}
\begin{proof}
We use the notations
\begin{align}\label{M2}
&S=\begin{pmatrix}
1 & P \\
0 &2
\end{pmatrix},\,\,
S^{-1}=\begin{pmatrix}
1 & -P/2 \\
0 &1/2
\end{pmatrix}.
\end{align}
We have
\begin{align}\label{Rep1}
\begin{pmatrix}
U_{mk+1}  \\
U_{mk}
\end{pmatrix}&=M^{mk}\begin{pmatrix}
1 \\
0
\end{pmatrix}=
(M^kSS^{-1})^{m-1}\begin{pmatrix}
U_{k+1}  \\
U_{k}
\end{pmatrix}.
\end{align}
By $(\ref{M1})$ and $(\ref{M2})$,
\beq\label{mkss} M^k S S^{-1}=\begin{pmatrix}
U_{k+1} & (-P U_{k+1}+V_{k+1})/2\\
U_{k} & (-P U_{k}+V_{k})/2
\end{pmatrix}.\eeq
By Lemma $\ref{L2}$ we can get the classical identity $ V_k=P U_k-2 Q U_{k-1}.$
With the help of which we eliminate $V_{k},V_{k+1}$ from $(\ref{mkss})$. Then
\beq\label{MSkSI}
M^k S S^{-1}=\begin{pmatrix}
U_{k+1} & - Q U_{k}\\
U_{k} & -Q U_{k-1}
\end{pmatrix}.
\eeq
Since $-Q U_{k-1}=U_{k+1}-P U_k$, we get
\beq\label{MSkSI1}
M^k S S^{-1}=\begin{pmatrix}
U_{k+1} & - Q U_{k}\\
U_{k} & U_{k+1}-P U_k
\end{pmatrix}.
\eeq
Finally, we can modify $(\ref{Rep1})$ into $(\ref{f02})$.
\end{proof}
If $m=2$ in $(\ref{f02})$, then we obtain the following identities:
\beq
\label{f03}
U_{2k}=U_{k}(2 U_{k+1}-P U_{k}),\eeq
\beq\label{f04}
U_{2k+1}=U_{k+1}^2-Q U_{k}^2.
\eeq
If  we replace $k$ by $k+1$ in $(\ref{f03})$ and use $U_{k+2}=PU_{k+1}-QU_k$, then we obtain
\beq\label{f05}
U_{2k+2}=U_{k+1}(P U_{k+1}-2Q U_{k}).
\eeq
Now using $(\ref{f03})$, $(\ref{f04})$, and $(\ref{f05})$  we can present an algorithm for computing two terms of $\{U_k(P,Q)\}$ with the indices $N$ and $N+1$. We need four temporary memories: $u_1, u_2, U_1, U_2$.

\smallskip
\noindent$\overline{\mbox{\underline{{\bf Algorithm 1} Computing the Lucas sequence $\{U_k(P,Q)\}$\quad\quad\quad \quad\quad\quad\quad\quad\quad\,\,\,\,\,\phantom{$\frac{1^2}{2}$}}}}$\\
 \noindent{\bf Input:} $N=\sum_{i=0}^{m-1}b_i 2^{i}$, $(b_{m-1}=1)$ \phantom{$\frac{\sqrt{A^7}}{4}$}\\
\noindent \phantom{{\bf Input:}} $P, Q$\\
\noindent{\bf Output:} $U_{N}$, $U_{N+1}$\\
\noindent \phantom{1}1: $U_1\leftarrow 1$; $U_2\leftarrow P$\\
\noindent \phantom{1}2: {\bf for} $j$ from $m-2$ to $0$ by $-1$ {\bf do}\\
\noindent \phantom{1}3: $u_1\leftarrow U_1$; $u_2\leftarrow U_2$\\
\noindent \phantom{1}4: \,\,\,\,\,\,\,\,\,\,{\bf if} $b_j=1$ {\bf then}\\
\noindent \phantom{1}5: \,\,\,\,\,\,\,\,\,\phantom{{\bf else if}}\,\,\, $U_1\leftarrow u_2^2-Q u_1^2$; $U_2\leftarrow u_2(P u_2-2Qu_1)$\\
\noindent \phantom{1}6: \,\,\,\,\,\,\,\,\,\,{\bf else if}\\
\noindent \phantom{1}7: \,\,\,\,\,\,\,\,\,\phantom{{\bf else if}}\,\,\, $U_1\leftarrow u_1(2 u_2-P u_1)$; $U_2\leftarrow u_2^2-Q u_1^2$\\
\noindent \phantom{1}8: \,\,\,\,\,\,\,\,\,\,{\bf end if}\\
\noindent \phantom{1}9: {\bf end for}\\
\underline{\noindent 11: {\bf return} $U_1, U_2$\phantom{${\frac{9}{9}}$}\quad\quad\quad\quad\quad\quad\quad\quad\quad\quad\quad\quad\quad\quad\quad\quad\quad\quad\quad\quad\quad\quad\quad\quad
\quad\quad\quad\,\,\,}

\smallskip
\noindent Such a computational  method was discussed by Reiter in \cite{1.5}. Previously \cite{1.6}, it was proposed for the Fibonacci numbers.

Suppose we have computed $U_{N}$, $U_{N+1}$ by Algorithm $1$, then with the help of $(\ref{f07})$ we get $W_N$. Using $W_{N+1}=B U_{N+1}-A Q U_{N}$ we get $W_{N+1}$. Thus, in a general case to compute the terms $U_{N}$, $U_{N+1}$ and $W_{N}$, $W_{N+1}$ we need $3m$  multiplications\footnote{We imply that $P, Q$ are not large. So multiplications that involve them are similar to additions.}, here $m=\lfloor\log_2 N\rfloor+1$. But when $Q=1$ or more generally $Q=a^2$, we can slightly transform \mbox{Algorithm $1$} so that we need only $2m$ multiplications. Indeed, when $Q=1$, we replace  the expression $u_2^2 - u_1^2$ by $(u_2-u_1)(u_2+u_1)$ at steps $5, 7$. When $Q=a^2$, we use the formula $u_2^2-Q u_1^2=(u_2-a u_1)(u_2+a u_1)$.

\subsection{Comparison with other existing algorithms}
Currently, the main algorithm \cite{1.8} for quick computation of the Lucas sequence terms $U_N$, $V_N$   uses the following properties:
\begin{align}\nonumber &V_{2k+1}=V_{k+1}V_{k}-P Q^k,\,\, V_{2k}=V_{k}^2-2 Q^k,\\
 &U_{2k+1}=U_{k+1}V_{k}-Q^k,\phantom{P}\,\, U_{2k}=U_{k}V_{k}.
\end{align}
When $Q=\pm1$, the algorithm needs  $3 m$ multiplications. When  $Q\not =\pm1$  and without any assumptions about $N$, this algorithm needs $11m/2$ multiplications. We see that  Algorithm $1$ is more effective in the general case, but there is an important case when the algorithm offered in \cite{1.8} is better. This is so when we need to compute the term $V_N(P,1)$ or $V_N(P,-1)$. For $N=2^s(2d+1)$ the algorithm in \cite{1.8}  needs  $2\lfloor\log_2 (2d+1)\rfloor+s$ multiplications while Algorithm $1$ needs $2\lfloor\log_2 (2d+1)\rfloor+2s$. So in applications such as Lucas-based cryptosystem \cite{1.16} and Lucas-Lehmer-Risel primality test \cite{1.15} it is preferable to use the algorithm offered in \cite{1.8}.

Now  we compare Algorithm $1$ with Fiduccia's algorithm. The characteristic polynomial is $g(x)=x^2-Px+Q$. To compute $x^N\bmod{g(x)}$ Fiduccia's algorithm uses classical method of repeating  squaring. For an arbitrary linear polynomial $h(x)=-u_1 x+u_2$  we have $h^2(x)\bmod{g(x)}=-u_1(2u_2-P u_1)x +u_2^2-Q u_1^2$.   As is seen from above, we can use the formulas $(\ref{f03})$, $(\ref{f04})$ for modular polynomial squarings. Therefore, Algorithm $1$ together with the formula $(\ref{f07})$ is one way of implementing Fiduccia's algorithm for second-order linear recurrences, where is used the explicit formulas for modular polynomial squarings.

\section{\bf{Computation of third-order linear recurrences}}
We will follow the notation for third-order linear recurrences according to \cite{1.18}. The sequences $\{X_k(p,q,r)\}$, $\{Y_k(p,q,r)\}$, and $\{Z_k(p,q,r)\}$ are defined recursively  by
\beq\label{xyz}
f_{k+3}=p f_{k+2}+q f_{k+1}+r f_{k},
\eeq
with the initial values $X_0=0$, $X_1=0$, $X_2=1$, $Y_0=0$, $Y_1=1$, $Y_2=0$, $Z_0=1$, $Z_1=0$, $Z_2=0$.
Similar to $(\ref{M1})$ we have
\beq\label{3M}
\begin{pmatrix}
X_{k+2} & Y_{k+2} & Z_{k+2}\\
X_{k+1} & Y_{k+1} & Z_{k+1}\\
X_{k} & Y_{k} & Z_{k}
\end{pmatrix}=M^{k} S, \text{ where } M=\begin{pmatrix}
p & q & r\\
1 & 0 & 0\\
0 & 1 & 0
\end{pmatrix},\,\, S=\begin{pmatrix}
1 & 0 & 0\\
0 & 1 & 0\\
0 & 0 & 1
\end{pmatrix}.
\eeq
\begin{lemma}\label{3L2}
Let the sequence $\{W_k(a_0,a_1,a_2;p,q,r)\}$ be  defined by the relation \beq\label{rel} W_{k+3}=p W_{k+2}+q W_{k+1}+r W_{k},\eeq
with the initial values $W_0=a_0, W_1=a_1, W_2=a_2$. Then
\beq\label{fW3}W_{k}=a_2 X_{k} + (a_1q+a_0r)X_{k-1}+a_1r X_{k-2},\eeq
\beq\label{fW3.1}W_{k}=(a_2-a_1 p-a_0 q)X_k+(a_1-a_0 p)X_{k+1}+a_0 X_{k+2}.\eeq
\end{lemma}
\begin{proof}
We have:
\begin{align}\begin{pmatrix}
W_{k+2}  \\
W_{k+1}\\
W_{k}
\end{pmatrix}=
M^{k}
\begin{pmatrix}
a_2 \\
a_1 \\
a_0
\end{pmatrix}&=
a_2 M^k
\begin{pmatrix}
1 \\
0 \\
0
\end{pmatrix}+
a_1 M^k
\begin{pmatrix}
0 \\
1 \\
0
\end{pmatrix}+
a_0 M^k
\begin{pmatrix}
0 \\
0 \\
1
\end{pmatrix}\nonumber\\
&=a_2\begin{pmatrix}
X_{k+2} \\
X_{k+1} \\
X_{k}\end{pmatrix}+
a_1 M^{k-1}\begin{pmatrix}
q \\
0 \\
1\end{pmatrix}+
a_0 M^{k-1}
\begin{pmatrix}
r \\
0 \\
0
\end{pmatrix}\nonumber\\
&=a_2\begin{pmatrix}
X_{k+2} \\
X_{k+1} \\
X_{k}\end{pmatrix}+
(a_1q+a_0r)\begin{pmatrix}
X_{k+1} \\
X_{k} \\
X_{k-1}\end{pmatrix}+
a_1 M^{k-2}\begin{pmatrix}
r \\
0 \\
0\end{pmatrix}\nonumber\\
&=
\begin{pmatrix}
a_2 X_{k+2} + (a_1q+a_0r)X_{k+1}+a_1r X_{k}\\
a_2 X_{k+1} + (a_1q+a_0r)X_{k}+a_1r X_{k-1} \\
a_2 X_{k} + (a_1q+a_0r)X_{k-1}+a_1r X_{k-2}\end{pmatrix}.
\end{align}
So we obtain $(\ref{fW3})$. With the help of $X_{k-2}=(X_{k+1}-p X_{k}-q X_{k-1})/r$ and $X_{k-1}=(X_{k+2}-p X_{k+1}-q X_{k})/r$ we  get $(\ref{fW3.1})$.
\end{proof}
By Lemma $\ref{3L2}$ we get the following
\begin{align}\label{VviaU3}
Y_k=qX_{k-1}+rX_{k-2},\\
\label{VviaU3.0}Z_k=r X_{k-1},\\
Y_k=X_{k+1}-pX_k,\\
\label{VviaU3.1}Z_k=X_{k+2}-pX_{k+1}-qX_{k}.
\end{align}
\begin{theorem}\label{T2}
Let $\{X_{k}(p,q,r)\}$ be the third-order linear recurrence sequence with the initial values $X_0=0, X_1=0, X_2=1$. Then
\beq\label{f08}
\begin{pmatrix}
X_{mk+2}  \\
X_{mk+1}  \\
X_{mk}
\end{pmatrix}=\begin{pmatrix}
X_{k+2} & qX_{k+1}+rX_k& rX_{k+1}\\
X_{k+1} & X_{k+2}-pX_{k+1}& rX_{k}\\
X_{k} & X_{k+1}-pX_{k}& X_{k+2}-pX_{k+1}-qX_k
\end{pmatrix}^{m-1}\begin{pmatrix}
X_{k+2}  \\
X_{k+1}  \\
X_{k}
\end{pmatrix}.
\eeq\end{theorem}
\begin{proof}
According to $(\ref{3M})$, we have
\begin{align}\label{Rep13}
\begin{pmatrix}
X_{mk+2}  \\
X_{mk+1}  \\
X_{mk}
\end{pmatrix}&=M^{mk}\begin{pmatrix}
1 \\
0 \\
0
\end{pmatrix}=\left(M^k\right)^{m-1}\begin{pmatrix}
X_{k+2}  \\
X_{k+1}  \\
X_k
\end{pmatrix}.
\end{align}
Using $(\ref{VviaU3})-($\ref{VviaU3.1}$)$ we eliminate  $Y_{k}, Y_{k+1}, Y_{k+2}, Z_{k}, Z_{k+1}, Z_{k+2}$ from $M^k$. This may be done in such a way that $M^k$ will contain only $X_k$,  $X_{k+1}$, $X_{k+2}$. We obtain
\beq\label{MSkSI3}
M^k=\begin{pmatrix}
X_{k+2} & qX_{k+1}+rX_k& rX_{k+1}\\
X_{k+1} & X_{k+2}-pX_{k+1}& rX_{k}\\
X_{k} & X_{k+1}-pX_{k}& X_{k+2}-pX_{k+1}-qX_k
\end{pmatrix}.
\eeq
Finally, we can modify $(\ref{Rep13})$ into $(\ref{f08})$.
\end{proof}
If we put $m=2$ in $(\ref{f08})$, then we get the following formulas:
\begin{align}\label{X3}
X_{2k+2}=&X_{k+2}^2+X_{k+1}(q X_{k+1}+2rX_k),\\
\label{2k+1}X_{2k+1}=&r X_k^2+X_{k+1}(2 X_{k+2}-p X_{k+1}),\\
\label{X3.1}X_{2k}=&X_{k+1}^2+X_{k}(2 X_{k+2}-2p X_{k+1}-q X_k).
\end{align}
\noindent{\bf Remark.} If we  put $r=0$, $q=-Q$ in these formulas and subtract $1$ from all indices, then up to the substitution of $U$ for $X$  we obtain the identities  for second-order recurrences. It  follows from $X_{k+1}(P,-Q,0)=U_{k}(P,Q)$.

\noindent{\bf Remark.} If we calculate the remainder
\beq\label{f010}
(c_1 x^2 +c_2 x + c_3)^2\bmod x^3 - p x^2 - q x - r,
\eeq
then we obtain the formulas similar (but not the same) to  $(\ref{X3})-(\ref{X3.1})$ for  squaring of quadratic polynomials  modulo $g(x)=x^3 - p x^2 - q x - r$. They can be used in Fiduccia's algorithm for computing third-order recurrences.

To get an algorithm for computing $X_N$, $X_{N+1}$, $X_{N+2}$ similar to the binary exponentiation we need to be able  to compute $X_{2k}$, $X_{2k+1}$, $X_{2k+2}$, $X_{2k+3}$ using $X_k$, $X_{k+1}$, $X_{k+2}$. So we need another formula that helps us to compute $X_{2k+3}$. It can be obtained from $(\ref{2k+1})$ if we replace $k$ by $k+1$ and use  $X_{k+3}=p X_{k+2}+q X_{k+1}+r X_{k}$. It is as follows:
\beq\label{2k+3}
X_{2k+3}=r X_{k+1}^2+X_{k+2}(p X_{k+2}+2 q X_{k+1}+2 r X_{k}).
\eeq
Now we present an algorithm based on the formulas $(\ref{X3})-(\ref{X3.1})$, $(\ref{2k+3})$. We need to use six temporary memories.

\bigskip

\noindent$\overline{\mbox{\underline{{\bf Algorithm 2} Computing the third-order linear recurrence $\{X_k(p,q,r)\}$\quad\quad\quad\, \hfill\phantom{$\frac{1^2}{2}$}}}}$\\
 \noindent{\bf Input:} $N=\sum_{i=0}^{m-1}b_i 2^{i}$, $(b_{m-1}=1)$ \phantom{$\frac{\sqrt{A^7}}{4}$}\\
\noindent \phantom{{\bf Input:}} $p, q, r$\\
\noindent{\bf Output:} $X_{N}$, $X_{N+1}$, $X_{N+2}$\\
\noindent \phantom{1}1: $X_1\leftarrow 0$; $X_2\leftarrow 1$; $X_3\leftarrow p$\\
\noindent \phantom{1}2: {\bf for} $j$ from $m-2$ to $0$ by $-1$ {\bf do}\\
\noindent \phantom{1}3: $x_1\leftarrow X_1$; $x_2\leftarrow X_2$; $x_3\leftarrow X_3$\\
\noindent \phantom{1}4: \,\,\,\,\,\,\,\,\,\,{\bf if} $b_j=1$ {\bf then}\\
\noindent \phantom{1}5: \,\,\,\,\,\,\,\,\,\,\phantom{{\bf else if}} $X_1\leftarrow r x_1^2+x_2(2 x_3-p x_2)$; $X_2\leftarrow x_{3}^2+x_{2}(q x_{2}+2r x_1)$;\\
\noindent \phantom{15:} \,\,\,\,\,\,\,\,\,\,\phantom{{\bf else if}}  $X_3\leftarrow r x_{2}^2+x_{3}(p x_{3}+2 q x_{2}+2 r x_{1})$\\
\noindent \phantom{1}6: \,\,\,\,\,\,\,\,\,\,{\bf else if}\\
\noindent \phantom{1}7: \,\,\,\,\,\,\,\,\,\phantom{{\bf else if}} $X_1\leftarrow x_{2}^2+x_{1}(2 x_{3}-2p x_{2}-q x_1)$; $X_2\leftarrow r x_1^2+x_2(2 x_3-p x_2)$;\\
\noindent \phantom{17:} \,\,\,\,\,\,\,\,\,\phantom{{\bf else if}} $X_3\leftarrow x_{3}^2+x_{2}(q x_{2}+2rx_1)$\\
\noindent \phantom{1}8: \,\,\,\,\,\,\,\,\,\,{\bf end if}\\
\noindent \phantom{1}9: {\bf end for}\\
\underline{\noindent 11: {\bf return} $X_1, X_2, X_3$\phantom{${\frac{9}{9}}$}\quad\quad\quad\quad\quad\quad\quad\quad\quad\quad\quad\quad\quad\quad\quad\quad\quad\quad\quad\quad\quad\quad\quad
\quad\quad\,\,\,}

\noindent We will imply that multiplications by $p, q, r$ can be simulated by additions. Then algorithm $2$ needs $3m$ multiplications and $3m$ squarings. At the end of this section, we refer to some applications that use computation of remote terms of third-order linear recurrence sequences; see \cite{1.19, Gong, Adams, Cho}.
%
\section{\bf{Computation of fourth-order linear recurrences}}
Since this section is similar to the previous one, we give only the main formulas and the final algorithm.

The fourth-order linear recurrence $\{W_k(a_0,a_1,a_2,a_3; p_0,p_1,p_2,p_3)\}$ is defined \mbox{recursively}  by
\beq\label{xyz1}
f_{k+4}=p_0 f_{k+3}+p_1 f_{k+2}+p_2 f_{k+1}+p_3 f_{k},
\eeq
with the initial values $W_0=a_0, W_1=a_1$, $W_2=a_2$, $W_3=a_3$. Denote the sequence $\{W_k(0,0,0,1; p_0,p_1,p_2,p_3)\}$ by $\{X_k(p_0,p_1,p_2,p_3)\}$. The formulas which can be obtained by the matrix method are:
\begin{align}\label{basic4}
W_k=&a_3 X_k+(a_0 p_3+a_1p_2+a_2p_1)X_{k-1}+(a_1p_3+a_2p_2)X_{k-2}+a_2p_3X_{k-3},\\
\label{basic4.1}W_k=&(a_3-a_2p_0-a_1p_1-a_0p_2)X_k+(a_2-a_1p_0-a_0p_1)X_{k+1}+\nonumber\\
&(a_1-a_0 p_0)X_{k+2}+a_0X_{k+3}.
\end{align}

We will use $W_k(a_0,a_1,a_2,a_3)$ instead of $W_k(a_0,a_1,a_2,a_3; p_0,p_1,p_2,p_3)$. By $(\ref{basic4})$, $(\ref{basic4.1})$, and $(\ref{xyz1})$  we obtain the following
\begin{align}\label{VviaU4}
W_k(0,0,1,0)&=p_1X_{k-1}+p_2X_{k-2}+p_3X_{k-3},\\
\label{VviaU4.0}W_k(0,1,0,0)&=p_2X_{k-1}+p_3X_{k-2},\\
W_k(1,0,0,0)&=p_3X_{k-1},\\
W_k(0,0,1,0)&=-p_0 X_{k}+X_{k+1},\\
W_k(0,1,0,0)&=-p_1X_{k}-p_0X_{k+1}+X_{k+2},\\
\label{VviaU4.1}W_k(1,0,0,0)&=-p_2X_{k}-p_1X_{k+1}-p_0X_{k+2}+X_{k+3}.
\end{align}
For convenience, we use the notation $^{i}W_{k}$ for $W_k(a_0,a_1,a_2,a_3)$ with only one nonzero $a_i=1$. Then by the matrix method we get
\begin{align}\label{Finally4}
\begin{pmatrix}
X_{mk+3}  \\
X_{mk+2}  \\
X_{mk+1}  \\
X_{mk}
\end{pmatrix}=
\begin{pmatrix}
X_{k+3} & ^{2}W_{k+3}&^{1}W_{k+3}&^{0}W_{k+3}\\
X_{k+2} & ^{2}W_{k+2}& ^{1}W_{k+2}&^{0}W_{k+2}\\
X_{k+1} & ^{2}W_{k+1}& ^{1}W_{k+1}&^{0}W_{k+1}\\
X_{k} & ^{2}W_k& ^{1}W_{k}& ^{0}W_{k}
\end{pmatrix}^{m-1}\begin{pmatrix}
X_{k+3}  \\
X_{k+2}  \\
X_{k+1}  \\
X_{k}
\end{pmatrix}.
\end{align}
With the help of $(\ref{VviaU4})-(\ref{VviaU4.1})$ we transform the matrix in $(\ref{Finally4})$ and obtain
\begin{align}\label{Fin}
\begin{pmatrix}
X_{k+3}&p_1X_{k+2}+p_2X_{k+1}+p_3X_{k}&p_2X_{k+2}+p_3X_{k+1}&p_3X_{k+2}\\
X_{k+2}&X_{k+3}-p_0X_{k+2}&p_2X_{k+1}+p_3X_{k}&p_3X_{k+1}\\
X_{k+1}&X_{k+2}-p_0X_{k+1}&X_{k+3}-p_0X_{k+2}-p_1X_{k+1}&p_3X_{k}\\
X_{k}&X_{k+1}-p_0X_{k}&X_{k+2}-p_0X_{k+1}-p_1X_k&R_{4,4}
\end{pmatrix}.
\end{align}
Here, $R_{4,4}=X_{k+3}-p_0X_{k+2}-p_1X_{k+1}-p_2X_{k}$. If we put $m=2$ in $(\ref{Finally4})$, then after simplification we get the following formulas:
\begin{align}\label{X4}
X_{2k+3}=&X_{k+3}^2+X_{k+2}(p_1 X_{k+2}+2p_2X_{k+1}+2p_3X_k)+p_3 X_{k+1}^2,\\
\label{X4.0}
X_{2k+2}=&X_{k+2}(2X_{k+3}-p_0X_{k+2})+X_{k+1}(p_2X_{k+1}+2p_3X_k),\\
\label{X4.11}X_{2k+1}=&X_{k+2}^2+X_{k+1}(2X_{k+3}-2p_0X_{k+2}-p_1X_{k+1})+p_3 X_{k}^2,\\
\label{X4.1}X_{2k}=&X_{k+1}(2X_{k+2}-p_0X_{k+1})+X_{k}(2 X_{k+3}-2p_0 X_{k+2}-2p_1 X_{k+1}-p_2 X_k).
\end{align}
We also need the formula for $X_{2k+4}$. It can be obtained from $(\ref{X4.0})$ if we replace $k$ by $k+1$ and use $(\ref{xyz1})$ for $X_{k+4}$. It is as follows:
\beq\label{X44}
X_{2k+4}=X_{k+3}(p_0X_{k+3}+2p_1X_{k+2}+2p_2X_{k+1}+2p_3X_k)+X_{k+2}(p_2X_{k+2}+2p_3X_{k+1}).
\eeq

\bigskip
\noindent$\overline{\mbox{\underline{{\bf Algorithm 3} Computing the fourth-order linear recurrence $\{X_k(p_0,p_1,p_2,p_3)\}$\phantom{$\frac{1^2}{2}$}}}}$\\
 \noindent{\bf Input:} $N=\sum_{i=0}^{m-1}b_i 2^{i}$, $(b_{m-1}=1)$ \phantom{$\frac{\sqrt{A^7}}{4}$}\\
\noindent \phantom{{\bf Input:}} $p_0, p_1, p_2, p_3$\\
\noindent{\bf Output:} $X_{N}$, $X_{N+1}$, $X_{N+2}$, $X_{N+3}$\\
\noindent \phantom{1}1: $X_1\leftarrow 0$; $X_2\leftarrow 0$; $X_3\leftarrow 1$; $X_4\leftarrow p_0$\\
\noindent \phantom{1}2: {\bf for} $j$ from $m-2$ to $0$ by $-1$ {\bf do}\\
\noindent \phantom{1}3: $x_1\leftarrow X_1$; $x_2\leftarrow X_2$; $x_3\leftarrow X_3$; $x_4\leftarrow X_4$\\
\noindent \phantom{1}4: \,\,\,\,\,\,\,\,\,\,{\bf if} $b_j=1$ {\bf then}\\
\noindent \phantom{1}5: \,\,\,\,\,\,\,\,\,\,\phantom{{\bf else if}} $X_1\leftarrow x_{3}^2+x_{2}(2x_{4}-2p_0x_{3}-p_1x_{2})+p_3 x_{1}^2$; \\
\noindent \phantom{15:} \,\,\,\,\,\,\,\,\,\,\phantom{{\bf else if}} $X_2\leftarrow x_{3}(2x_{4}-p_0x_{3})+x_{2}(p_2x_{2}+2p_3x_1)$;\\
\noindent \phantom{15:} \,\,\,\,\,\,\,\,\,\,\phantom{{\bf else if}} $X_3\leftarrow x_{4}^2+x_{3}(p_1 x_{3}+2p_2x_{2}+2p_3x_1)+p_3 x_{2}^2$;\\
\noindent \phantom{15:} \,\,\,\,\,\,\,\,\,\,\phantom{{\bf else if}} $X_4\leftarrow x_{4}(p_0x_{4}+2p_1x_{3}+2p_2x_{2}+2p_3x_1)+x_{3}(p_2x_{3}+2p_3x_{2})$\\
\noindent \phantom{1}6: \,\,\,\,\,\,\,\,\,\,{\bf else if}\\
\noindent \phantom{1}7: \,\,\,\,\,\,\,\,\,\phantom{{\bf else if}} $X_1\leftarrow x_{2}(2x_{3}-p_0x_{2})+x_{1}(2 x_{4}-2p_0 x_{3}-2p_1 x_{2}-p_2 x_1)$;\\
\noindent \phantom{17:} \,\,\,\,\,\,\,\,\,\phantom{{\bf else if}} $X_2\leftarrow x_{3}^2+x_{2}(2x_{4}-2p_0x_{3}-p_1x_{2})+p_3 x_{1}^2;$\\
\noindent \phantom{17:} \,\,\,\,\,\,\,\,\,\phantom{{\bf else if}} $X_3\leftarrow x_{3}(2x_{4}-p_0x_{3})+x_{2}(p_2x_{2}+2p_3x_1);$\\
\noindent \phantom{17:} \,\,\,\,\,\,\,\,\,\phantom{{\bf else if}} $X_4\leftarrow x_{4}^2+x_{3}(p_1 x_{3}+2p_2x_{2}+2p_3x_1)+p_3 x_{2}^2$\\
\noindent \phantom{1}8: \,\,\,\,\,\,\,\,\,\,{\bf end if}\\
\noindent \phantom{1}9: {\bf end for}\\
\underline{\noindent 11: {\bf return} $X_1, X_2, X_3, X_4$\phantom{${\frac{9}{9}}$}\quad\quad\quad\quad\quad\quad\quad\quad\quad\quad\quad\quad\quad\quad\quad\quad\quad\quad\quad\quad\quad\quad
\quad\,\,\,\,}

\noindent As is seen from the algorithm we need $6m$ multiplications and $4m$ squarings to compute the terms $X_N$, $X_{N+1}$, $X_{N+2}$, $X_{N+3}$. Here, as in the previous section, we count only ``big" multiplications. 
\section{\bf{Computation of linear recurrence sequences of  any order}}
Let $\{W_k(a_0,\ldots,a_{n-1};p_0\ldots p_{n-1})\}$ be an $n{\it th}$-order linear recurrence defined by the relation  $f_{k+n}=\sum_{i=0}^{n-1}p_if_{k+n-1-i}$, with the initial values $W_i=a_i$ $(0\leq i\leq n-1)$. Let $\{X_k(p_0,\ldots,p_{n-1})\}$ be the sequence that is derived from $\{W_k\}$ if $a_{n-1}=1$ and the other $a_i=0$. Using the matrix method as in Lemma $\ref{3L2}$ and by mathematical induction we get the following formulas
\begin{align}\label{Ind}
W_k&=a_{n-1}X_k+\sum_{j=1}^{n-1}\(p_j\sum_{i=0}^{j-1}a_{n-2-i}X_{k-j+i}\),\\
\label{Ind2}W_k&=\sum_{j=0}^{n-1}\(a_{n-1-j}-\sum_{i=0}^{n-j-2}a_{n-j-2-i}p_i\)X_{k+j}.
\end{align}
If we put $n=4$ in these formulas, then we obtain $(\ref{basic4})$ and $(\ref{basic4.1})$.

Repeating the arguments of the previous section we get the matrix formula
\begin{align}\label{Matr}
\begin{pmatrix}
X_{2k+n-1}  \\
\vdots  \\
X_{2k+1}  \\
X_{2k}
\end{pmatrix}=
\begin{pmatrix}
X_{k+n-1} & ^{n-2}W_{k+n-1}&^{n-3}W_{k+n-1}&\dots&^{0}W_{k+n-1}\\
\vdots  & \vdots     & \vdots &\ddots   &\vdots\\
X_{k+1} & ^{n-2}W_{k+1}& ^{n-3}W_{k+1}&\dots&^{0}W_{k+1}\\
X_{k} & ^{n-2}W_k& ^{n-3}W_{k}&\dots& ^{0}W_{k}
\end{pmatrix}\begin{pmatrix}
X_{k+n-1}  \\
\vdots  \\
X_{k+1}  \\
X_{k}
\end{pmatrix}.
\end{align}
Here, as above $^{i}W_{k}$ denotes  $W_k(a_0,\ldots,a_{n-1})$ with only one nonzero $a_i=1$.
Let $R=\(r_{i,j}\)$ be the matrix from $(\ref{Matr})$. It has a special form, see   $(\ref{f08})$ and $(\ref{Fin})$. Note that if we know two rows of the matrix $R$ which have numbers of different parity, then we can get the other rows. For example, we assume we know a formula which relates $X_{2k+\ell}$ to $X_{k+i}$ \mbox{$(0\leq i\leq n-1)$}, in other words we know the $(n-\ell)${\it th} row. If we replace $k$ by $k+1$ and use $X_{k+n}=\sum_{i=0}^{n-1}p_i X_{k+n-1-i}$, then we get the formula for  \mbox{$X_{2k+\ell+2}$} that corresponds to the \mbox{$(n-\ell-2)${\it th} row}. Repeating this procedure we obtain all rows with numbers  of the same parity as the parity of the \mbox{$(n-\ell)${\it th} row}. Thus, to get all formulas that will be used in the algorithm, we  need to know formulas for $X_{2k}$, $X_{2k+1}$.

For the elements of $R$ using $(\ref{Ind})$, $(\ref{Ind2})$ we obtain
\begin{align}\label{R}r_{i,j}=
\begin{cases}
X_{k+n-1-(i-j)}-\sum_{l=0}^{j-2}p_{l}X_{k+n-2-l-(i-j)}, \text{ if $i\geq j$},\\
\sum_{l=j-1}^{n-1}p_{l}X_{k+n-2-l-(i-j)}, \text{ if $i<j$}.\end{cases}\\
\nonumber\end{align}
Also, from the two last rows in $(\ref{Matr})$ we obtain the formulas which relate $X_{2k}$, $X_{2k+1}$ to $X_{k+i}$ \mbox{$(0\leq i\leq n-1)$}. These formulas are of the same form as  $(\ref{X4.11})$, $(\ref{X4.1})$.
\begin{align}\label{X(2k)}
X_{2k} &=e X_{k+(n-1)/2}^2+\nonumber\\&\sum_{i=0}^{\lfloor v\rfloor}X_{k+\lfloor v\rfloor-i} \left(2X_{k+\lceil v\rceil+1+i}-
p_{2i+e}X_{k+\lfloor v\rfloor-i}-2\sum_{j=0}^{2i-1+e}p_{j}X_{k+\lceil v\rceil+i-j}\right),
\end{align}
\begin{align}\label{X(2k+1)}
&X_{2k+1}=p_{n-1}X_{k}^2+(1-e) X_{k+n/2}^2+\nonumber\\&\sum_{i=0}^{\lceil v\rceil-1}X_{k+\lceil v\rceil-i} \left(2X_{k+\lfloor v\rfloor+2+i}-
p_{2i+1-e}X_{k+\lceil v\rceil-i}-2\sum_{j=0}^{2i-e}p_{j}X_{k+\lfloor v\rfloor+1+i-j}\right).
\end{align}
Here,  $e=n\bmod 2$ and $v=n/2-1$. 

The scheme for computing terms of $\{W_{k}(a_0,\ldots,a_{n-1};p_0,\ldots,p_{n-1})\}$ is:

\noindent $(i)$
Using $X_{k+n}=\sum_{i=0}^{n-1}p_i X_{k+n-1-i}$ and repeating  the replacement of $k$ by $k+1$ in $(\ref{X(2k)})$, $(\ref{X(2k+1)})$ without removing brackets we  obtain the formulas for   $X_{2k+i}$ $(0\leq i \leq n)$. 
These formulas determine the rules of transition from the  terms $X_{k+i}$ $(0\leq i\leq n-1)$ to $X_{2k+i}$ $(0\leq i\leq n-1)$ and also to $X_{2k+1+i}$ $(0\leq i\leq n-1)$.
\\
\noindent $(ii)$  By using these formulas we obtain an algorithm for computing $\{X_k\}$ that is similar to Algorithm $3$.
\\
\noindent $(iii)$ To get the value $W_N$ we need to use $(\ref{Ind2})$ after we have computed $X_{N+i}$ $(0\leq i\leq n-1)$ by the algorithm in $(ii)$.
\\
\noindent $(iv)$ In order to obtain $W_{N+1}$ we use the recurrence relation to get $X_{N+n}$ from $X_{N+i}$ $(0\leq i\leq n-1)$ and use $(\ref{Ind2})$.

\smallskip
\noindent{\bf Remark.} To compute the $N${\it th} term of an $n${\it th}-order linear recurrence we need $n(n+1)/2\log_2N$  multiplications\footnote{We use such a complexity model that multiplications involving $p_i$ are similar to additions.}. Indeed,  when $n$ is even, the formulas for  $X_{2k+2i}$ $(0\leq i\leq n/2)$ contain $n/2$  multiplications\footnote{Since they were derived from $(\ref{X(2k)})$ without removing brackets.} and for  $X_{2k+2i+1}$ $(0\leq i\leq n/2-1)$ contain $n/2+1$ multiplications. It is easy to see that each step of the algorithm needs $n/2$ formulas of the first type and $n/2$ formulas of the second type. Then to compute $X_{2k+i}$ $(0\leq i\leq n-1)$ or $X_{2k+1+i}$ $(0\leq i\leq n-1)$ using $X_{k+i}$ $(0\leq i\leq n-1)$ we need $n(n+1)/2$ multiplications. Thus, computing $X_{N+i}$ $(0\leq i\leq n-1)$ needs $n(n+1)/2\log_2N$ multiplications. Since $(\ref{Ind2})$ does not contain  ``big" multiplications, the above statement is proved for even $n$. The proof for odd $n$ by analogous.

Finally, we give the implementation of the above scheme in {\tt Mathematica}\footnote{Version Number: 10.4.0.0.}. The  function $\text{{\bf AnyOrderRecurrence}}[a,p,N]$ returns $W_N(a_0,\ldots,a_{n-1};p_0,\ldots,p_{n-1})$, where $N$ is a positive integer and $a$, $p$ are strings of length $n$.
\newpage \begin{adjustwidth}{-10em}{-10em}
\begin{center}
 \includegraphics[scale=0.7]{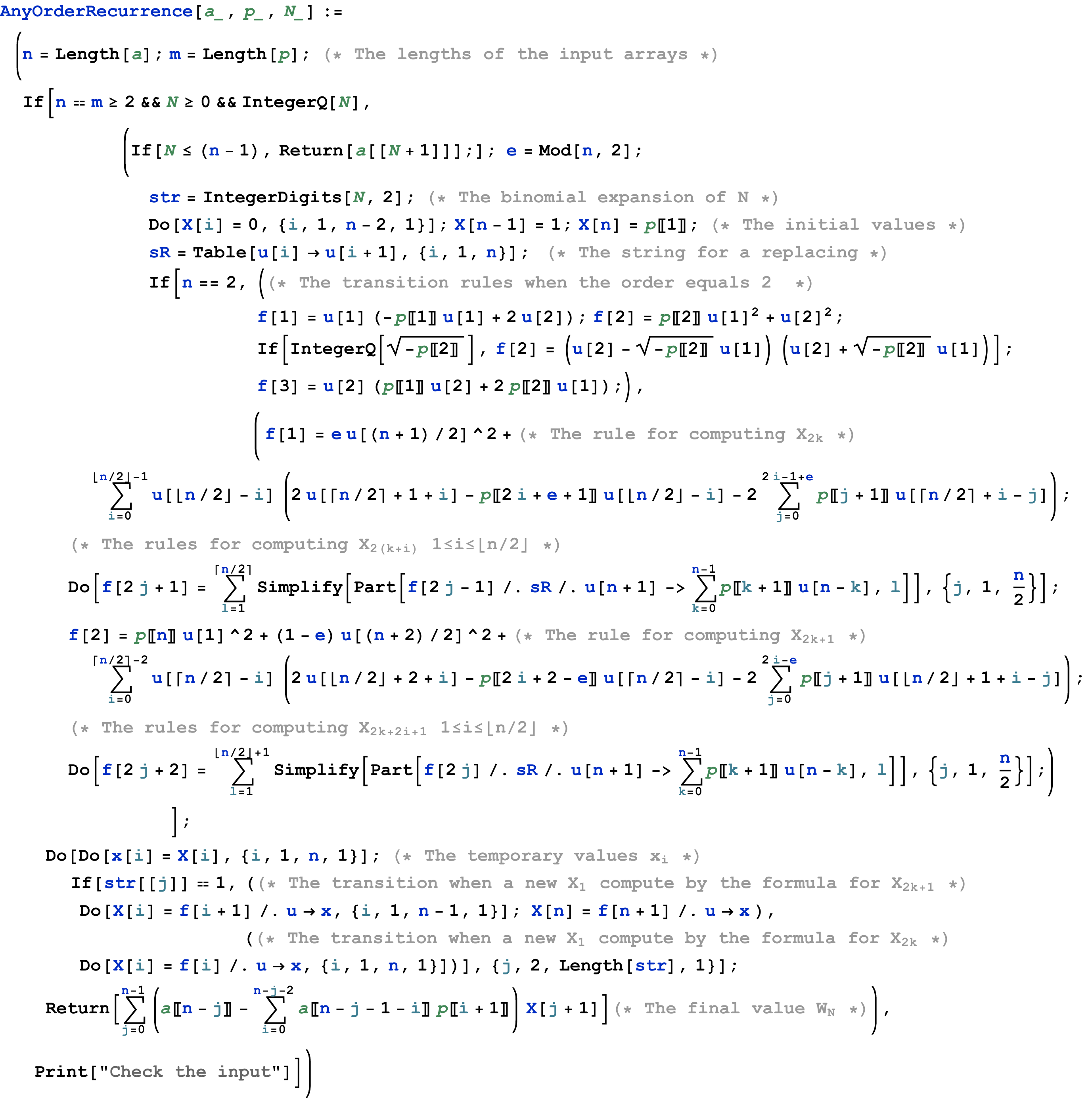}
 \end{center}\end{adjustwidth}

\noindent{\bf Acknowledgments.} The author is very grateful to A. Bostan for pointing to the reference \cite{0} and for the evidence that our algorithm  is one particular way of implementing Fiduccia's algorithm, where modular polynomial squarings are hard-coded.

\medskip
\end{document}